\documentclass[11pt]{amsart}
\usepackage{amsmath,amsfonts,amsthm,amssymb,amscd}
\usepackage{mathrsfs, textcomp}

\usepackage{textgreek}
\usepackage{color}
\usepackage{enumitem}

\newtheorem{theorem}{Theorem}[section]
\newtheorem{lemma}[theorem]{Lemma}
\newtheorem{proposition}[theorem]{Proposition}
\newtheorem{corollary}[theorem]{Corollary}

\theoremstyle{definition}

\newtheorem{definition}[theorem]{Definition}

\newtheorem{example}[theorem]{Example}

\newtheorem{remark}[theorem]{Remark}

\newcommand{\Chi}{{\raise2pt\hbox{$\chi$}}}
\newcommand{\ssm}{\smallsetminus}

\newcommand{\ra}{\rightarrow}

\newcommand{\N}{\mathbb{N}}

\newcommand{\R}{\mathbb{R}}

\newcommand{\Z}{\mathbb{Z}}

\newcommand{\fJ}{{\mathfrak J}}
\newcommand{\fN}{{\mathfrak N}}

\newcommand{\cI}{\mathcal{I}}

\newcommand{\cP}{\mathcal{P}}

\newcommand{\Prufer}{Pr\"{u}fer }

\newcommand{\ol}{\overline}
\DeclareMathOperator{\Spec}{Spec}
\DeclareMathOperator{\Max}{Max}

\DeclareMathOperator{\Min}{Min}
\DeclareMathOperator{\Ann}{Ann}

\DeclareMathOperator{\reg}{reg}
\DeclareMathOperator{\areg}{areg}

\newcommand\medcap{{\,\raise1.5pt\hbox{{$\scriptstyle{\bigcap}$}}\,}}
\newcommand\medcup{{\,\raise1.5pt\hbox{{$\scriptstyle{\bigcup}$}}\,}}

\usepackage{stackengine}

\title{Complemented Elements in Commutative Rings}

\author[P. Bhattacharjee]{Papiya Bhattacharjee}
\address{Department of Mathematics and Statistics, Florida Atlantic University, Boca Raton, FL 33431, USA}
\email{pbhattacharjee@fau.edu}

\author[W.Wm. McGovern]{Warren Wm. McGovern}
\address{H. L. Wilkes Honors College, Florida Atlantic University, 5353 Parkside Dr., Jupiter, FL 33458}
\email{warren.mcgovern@fau.edu }

\author[Y. Zhou]{Yiqiang Zhou}
\address{Department of Mathematics and Statistics, Memorial University of Newfoundland, St.John's, NL A1C 5S7, Canada}
\email{zhou@mun.ca}



\begin{document}

\begin{abstract}
In what follows we generalize the notion of a complemented ring to rings that are not necessarily reduced. We then determine how our concepts fit in with other well-known classes of rings.

\end{abstract}
\maketitle
\thispagestyle{empty}

\section{Introduction}

Throughout, we assume that all rings are commutative with identity different than 0. For a ring $R$, the classical (aka total) ring of quotients of $R$ is denoted by $q(R)$. We denote the space of prime ideals, minimal prime ideals, and maximal prime ideals of $R$ by $\Spec(R), \Min(R)$, and $\Max(R)$, respectively. When equipped with the hull-kernel topology, $\Min(R)$ becomes a zero-dimensional Hausdorff space, while $\Max(R)$ is a compact space.

We denote the set of units and regular elements by $U(R)$ and $\reg(R)$, respectively. We let $\fN(R)$ and $\fJ(R)$ denote the nilradical and Jacobson radical of $R$. Recall that $\fN(R)$ equals the intersection of all minimal prime ideals.  There has been a lot of interest over the last couple of decades in determining when $q(R)$ is a von Neumann regular ring. It is well-known that this occurs if and only if $R$ has the property that for each $a\in R$ there is a $b\in R$ such that $ab=0$ and $a+b\in \reg(R)$. Such an element $a\in R$ is said to be {\bf complemented}, and the pair $a,b$ is called a complementary pair. The ring $R$ is said to be complemented when every element of $R$ is complemented. We observe that complemented rings are necessarily reduced, due to the fact that the only nilpotent complemented element is 0.

The structure of the paper is as follows. In the remainder of this section we share some results and examples related to complemented elements and rings. In Section 2, we define the class of semi-complemented rings and compare this class to some other classes of rings that have been studied before. The third section is for studying products of these different kinds of rings. In the fourth section, we explore how to construct rings having the different properties.

Here is the main theorem that characterizes complemented rings.

\begin{theorem}\cite[Theorem 4.5]{huckaba}\label{complemented}
Suppose $R$ is a commutative ring with identity. The following statements are equivalent.
\begin{enumerate}[label={\rm \arabic*.}, nolistsep]
\item $R$ is complemented.
\item $q(R)$ is von Neumann regular.
\item $\Min(R)$ is compact and $R$ has Property $A$.
\item $\Min(R)$ is compact and $R$ satisfies the a.c.
\end{enumerate}
\end{theorem}

In the notes for Section 4 of \cite{huckaba}, a nice account is given of the proofs of Theorem \ref{complemented}.

\begin{definition}
(1) $R$ is said to have {\bf Property $A$} if every finitely generated dense ideal contains a regular element.

(2) $R$ is said to satisfy the {\bf annihilator condition} (or a.c.) if for every $a,b\in R$ there is a $c\in R$ such that $\Ann(a,b)=\Ann(c)$.
\end{definition}

Observe that it follows that a von Neumann regular ring is complemented. It is well-known that the von Neumann regular rings are precisely the zero-dimensional reduced rings. Element-wise we have the following definition.

\begin{definition}
Let $x\in R$. We say that $x$ is a {\bf von Neumann regular} element if there is an $a\in R$ such that $x^2a=x$.
\end{definition}

Of course, the ring $R$ is von Neumann regular when every element is von Neumann regular. Now, to understand zero-dimensionality we recall that an element $x\in R$ is {\bf $\pi$-regular} if some power of $x$ is von Neumann regular. We recall a characterization of  zero-dimensional rings.

\begin{theorem}\cite[Theorem 6]{arapovic} The following statements are equivalent for a ring $R$.
\begin{enumerate}[label={\rm \arabic*.}, nolistsep]
\item The ring $R$ is zero-dimensional.
\item Every element of $R$ is $\pi$-regular.
\item $R/\fN(R)$ is von Neumann regular.
\end{enumerate}
\end{theorem}
\vspace{.1in}

\begin{lemma}
Let $R$ be a ring and $x\in R$. If $x$ is a von Neumann regular element, then $x$ is complemented.
\end{lemma}

\begin{proof}
Choose $a\in R$ such that $x^2a=x$. Then $ax$ is idempotent, and hence so is $1-ax$. We claim that $1-ax$ is a complement of $x$.
Clearly,
$$x(1-ax)= x-x^2a=0.$$
Next, we show that $x+(1-ax)\in \reg(R)$. Suppose that $s(x+1-ax)=0$. Then $sx=s(ax-1)$, and so
$$sx = sx(ax)=s(ax-1)(ax) = 0.$$
This also means that $s(1-ax)=0$ and $s(ax)=0$. But
$$s=s(ax+ 1- ax) =0.$$

\end{proof}

Now, as we pointed out before, a non-zero nilpotent element cannot be complemented and therefore, since nilpotents are $\pi$-regular it follows that not every $\pi$-regular element is complemented.

We end the section with some examples of complemented rings in the hopes of demonstrating the importance of this class of rings.

\begin{example}
Let $X$ be a completely regular Hausdorff space and let $C(X)$ denote the ring of real-valued continuous functions equipped with pointwise addition and multiplication. The ring $C(X)$ is always reduced and in two independent articles (\cite{ls} and \cite{hw}) the question of when $C(X)$ is a complemented ring was answered. The ring $C(X)$ is complemented if and only if $X$ is a cozero-complemented space. We don't feel the need to get into the topological discussion, but we would like to add that any metric space is cozero-complemented. In particular, $C(\R)$ is a complemented ring. On the other hand, if $D$ is an uncountable discrete space and $\alpha D$ is its one-point compactification, then $C(\alpha D)$ is not complemented. For more information on $C(X)$ see \cite{gj}.
\end{example}

\begin{example}
In \cite{epstein}, the author studies the $A+B$ construction and characterizes when this ring is complemented. Recall that one starts with a collection of prime ideals of $A$, say $\cP$, with the property that $\cap \cP=\{0\}$. Taking the index set $\cI=\cP\times \N$ and for each $i=(P,n)$ we let $P_i=P$ and let $B$ be the direct sum $\oplus_{i\in I} A/P_i$. We take the ring $A$ and view it as a subring of $\Pi_{i\in I} A/P_i$ and then $A+B$ is the ring whose elements are of the form $a+b$ for some $a\in A$ and $b\in B$. Then $A+B$ is complemented if and only if $A$ is complemented and $\cP\subseteq \Min(A)$. One can also form $A+Q(B)$ where $Q(B)=\oplus_{i\in I} q(A/P_i)$, and the result extends to $A+Q(B)$.

\end{example}

\begin{example}
Recall that McCoy's Theorem states that for any commutative ring $R$, the polynomial ring $R[x]$ always satisfies Property $A$. Therefore, $R[x]$ is complemented if and only if $\Min(R)$ is compact.

There are rings that have Property $A$ that do not satisfy the a.c., and vice-versa (see \cite{lucas86}). Quentel's example \cite[Example 6, Section 27]{huckaba} has a compact minimal prime spectrum while not satisfying the a.c. Hence, Quentel's example $Q$ does not have Property $A$, nor is its classical ring of quotients von Neumann regular. But $Q[x]$ is complemented.
\end{example}
\vspace{.3in}

\section{Complemented Elements}

So far it should be apparent that all regular elements as well as von Neumann regular elements are complemented, and that non-zero nilpotent elements are not complemented. It is this latter fact that is one of the main motivations for studying semi-complemented rings. Here are some ways to create complemented elements.

\begin{lemma}\label{MainLemma}
Let $R$ be a ring, $a,a',b\in R$, and $s\in \reg(R)$.
\begin{enumerate}
\item Suppose $x\in\fN(R)$. Then $x$ is complemented if and only if $x=0$.
\item The sum and difference of a regular and a nilpotent is regular.
\item The element $a$ is complemented with complement $b$ if and only if $sa$ is complemented with complement $b$.
\item If $a, a'\in R$ are complemented with a common complement $b$, then $aa'$ is complemented with complement $b$.
\item If $a$ is complemented with complement $b$, then, for all $n\ge 1$ and $m\ge 1$, $a^n$ is complemented with  complement $b^m$.
\item If $b$ is a complement of $a$ and $a+bb'$ is complemented with $bb'\in {\rm nil}(R)$, then $bb'=0$.
\end{enumerate}
\end{lemma}

\begin{proof}
(1) Clearly, $0$ is complemented. Let $x\in \fN(R)$ and suppose that $y\in R$ satisfies $xy=0$ and $x+y\in\reg(R)$. Choose $n\in \N$ such that $x^n=0$. Then, by the Binomial Theorem and using that $xy=0$,
$$(x+y)^n = x^n+y^n=y^n,$$
whence $y\in \reg(R)$. It follows by regularity that $x=0$.
\vspace{.1in}

(2) Let $r\in \reg(R)$ and $x\in \fN(R)$. Let $n\in \N$ be such that $x^n=0$. Observe that
$$(r-x)(r^{n-1}+xr^{n-2}+x^2r^{n-3}+\cdots + x^{n-2}r+x^{n-1}) = r^n-x^n =r^n.$$
Since any factor of a regular ring is regular, we gather that $r-x\in\reg(R)$.
\vspace{.1in}

(3) Suppose $a$ is complemented with complement $b$. Then since $(sa)(sb)=0$ and $sa+sb=s(a+b)\in\reg(R)$ it follows that $sa$ is complemented with $sb$ as its complement.
\vspace{.1in}

Conversely, suppose that $sa$ is complemented with complement $b$. We show that $a$ is complemented with $b$ as a complement. Now, $s(ab)=(sa)b=0$ and so by regularity $ab=0$. Now, let $y\in R$ satisfy $y(a+b)=0$; $ya=-yb$ and so $ya^2=0=yb^2$.
Then,
$$y(sa+b)^2= ys^2a^2+2ysab+yb^2=0;$$
whence $y=0$ since $(sa+b)^2\in \reg(R)$.
\vspace{.1in}

(4) Suppose that $b$ is a common complement for $a,a'$. Then $(aa')b=0$ so it suffices to show that $aa'+b\in \reg(R)$. To that end, suppose that $(aa'+b)z=0$. Then $aa'z=-bz$, and so multiplying both sides by $b$ yields that $b^2z=0$. Since $(a+b)bz=b^2z=0$ it follows that $bz=0$, and hence also $aa'z=0$. Now, $aa'+b^2=(a+b)(a'+b)\in \reg(R)$, and so $(aa'+b^2)z=0$. Consequently, $aa'+b\in\reg(R)$
\vspace{.1in}

(5) This follows from repeated uses of (4).
\vspace{.1in}

(6) Let $z$ be a complement of $a+bb'$; so $(a+bb')z = 0$ and $a+bb'+z\in \reg(R)$. Since $bb' \in\fN(R)$, $a+z\in \reg(R)$ by (2). From $az = -bb'z$, we have $a^2z = 0$. Thus, $(a + b)az = 0$. By regularity, $az = 0$ and so also $bb'z = 0$. Hence $(a + z)bb' = 0$. Once again by regularity, $bb' = 0$.
\end{proof}
\vspace{.1in}

In \cite{branca}, the author, under the direction of the second author of the current paper, began the study of a generalization of complemented rings to rings with non-zero nilpotent elements. The idea behind the generalization had its roots in the article \cite{bmz} where generalized unit fusible rings were studied, a concept that deals with looking at elements outside of the Jacobson radical. Here are two variants of complemented rings.
\vspace{.1in}

\begin{definition}
1) The ring $R$ is said to be {\bf semi-complemented} if every non-nilpotent element is complemented. In other words, for every $x\in R\ssm \fN(R)$ there exists a $y\in R$ such that $xy=0$ and $x+y\in \reg(R)$.
\vspace{.1in}

2) The ring $R$ is said to be {\bf almost complemented} if $R/\fN(R)$ is a complemented ring.

\end{definition}
\vspace{.1in}

It follows that the ring $R$ is complemented if and only if it is a reduced semi-complemented ring if and only if it is a reduced almost complemented ring. It should be apparent that we shall not assume that our rings are reduced. Recall that an ideal $I$ of $R$ is said to be {\bf nil} if it is contained in the nilradical, and {\bf nilpotent} if some power of it is 0.

The proof of the following is straightforward. We shall use the ``bar" notation to denote the coset of an element in the factor ring $R/\fN(R)$.

\begin{lemma}\label{lem-a.c.}
The ring $R$ is almost complemented if and only if for every $x\in R\ssm\fN(R)$, there is a $y\in R$ such that $xy\in\fN(R)$ and $\overline{x+y}\in \reg(\overline{R})$. Moreover, if $R$ is almost complemented, then the only if part is true for all $x$.
\end{lemma}

\begin{proposition}
If $R$ is semi-complemented, then $R$ is almost complemented.
\end{proposition}

\begin{proof}
Suppose $R$ is semi-complemented. In order to show that $R$ is almost complemented, let $\overline{r}\in \ol{R}$. Without loss of generality, we assume that $r\notin \fN(R)$. By hypothesis, there is an $s\in R$ such that $rs=0$ and $r+s\in \reg(R)$. It follows that
$$\ol{r} \cdot \ol{s}=\ol{rs}=\ol{0} $$
and
$$\ol{r}+\ol{s}=\ol{r+s} \in \reg(\ol{R}).$$
Therefore, $\ol{R}$ is complemented, i.e. $R$ is almost complemented.
\end{proof}

Our next result generalizes what is known for complemented rings.
\begin{proposition}\label{q(R)}
The ring $R$ is semi-complemented if and only if its classical ring of quotients $q(R)$ is semi-complemented.
\end{proposition}

\begin{proof}
Recall that if $\frac{a}{s}\in q(R)$ with $a\in R$ and $s\in \reg(R)$, then $\frac{a}{s}\in\fN(q(R))$ if and only if $a\in \fN(R)$.
\vspace{.1in}

Suppose that $R$ is semi-complemented and let $x\in q(R)\ssm \fN(q(R))$ and write $x=\frac{a}{s}$. Then $a\in R\ssm \fN(R)$ so that by hypothesis there is some $b\in R$ such that $ab=0$ and $a+b\in\reg(R)$. Set $y=\frac{b}{s}\in q(R)$. Then $xy=0$ and $x+y=\frac{a+b}{s}\in \reg(q(R))$. Thus, $q(R)$ is semi-complemented.
\vspace{.1in}

Conversely, suppose that $q(R)$ is semi-complemented and let $r\in R\ssm\fN(R)$. Then $\frac{r}{1}\in q(R)\ssm\fN(q(R))$ and so there is a complement $\frac{b}{s}$. Since $\frac{r}{1}\frac{b}{s}=0$, it follows that $rb=0$. Since $\frac{sr+b}{s}=r+\frac{b}{s}\in\reg(q(R))$ we gather that $sr+b\in \reg(R)$. Thus, $sr$ is complemented in $R$. Then (3) of Lemma \ref{MainLemma} assures us that $r$ is complemented in $R$. Consequently, $R$ is semi-complemented.
\end{proof}

\vspace{.1in}

\begin{remark}\label{remark-reg}
Now it should be obvious that if $r\in \reg(R)$, then $\ol{r}\in \reg(\ol{R})$. We find it interesting that the converse need not hold. This opens up some possibilities. In our research, we found that others have considered similar ideas; we put some of the previous work into its proper context.
\end{remark}
\vspace{.2in}

In the article \cite{ab}, the authors were interested in classifying when $q(R)$ is a zero-dimensional ring. We would like to add some new ways to characterize the situation. We would also like to put this class into its proper place between semi-complemented rings and almost complemented rings. We need to develop some theory first. In studying the difference between almost complemented and semi-complemented we have found the following type of element plays a pivotal role.

\begin{definition}
We let
$$\areg(R)=\{x\in R: \bar{x}\in \reg(\bar{R})\},$$ the set of elements whose factors are regular in $R/\fN(R)$. We call any element belonging to $\areg(R)$ {\bf aregular}. Notice that any aregular element is not nilpotent. Internally, an element $x\in R$ belongs to $\areg(R)$ precisely if $xa\in\fN(R)$ implies $a\in\fN(R)$. Observe that $\areg(R)$ is a saturated multiplicative subset of $R$.

By Lemma \ref{lem-a.c.}, we have the following reformulation: $R$ is almost complemented if and only if for every $x\in R$, there is a $y\in R$ such that $xy\in\fN(R)$ and $x+y\in\areg(R)$. This also leads to the observation, that we can define an element $a\in R$ to be {\bf almost complemented} if there is a $b\in R$ such that $ab\in\fN(R)$ and $a+b\in \areg(R)$. Right away, we can point out that all aregular elements and all nilpotent elements are almost complemented. With this definition in hand, a ring $R$ is almost complemented if and only if every element is almost complemented if and only if every non-nilpotent element is almost complemented.
\end{definition}

\begin{remark}
It has been pointed out to us that in a current project by Dube and Ighetedo \cite{di}, the authors define an element $x$ to be {\it prime to an ideal $I$} if $xa\in I$ implies $a\in I$. It can be shown that $\reg(R)$ and $\areg(R)$ are the elements that are prime to $\{0\}$ and $\fN(R)$, respectively.
\vspace{.1in}

We can describe the elements of $\areg(R)$ topologically. Recall that on the prime spectrum $\Spec(R)$, a basic open set for hull-kernel topology has the form $U(a)=\{P\in\Spec(R): a\notin P\}$. The collection $\{U(a):a\in R\}$ forms a base for this topology on $\Spec(R)$. Furthermore, we leave it to the interested reader to show that $x\in\areg(R)$ if and only if $U(a)$ is dense subset of $\Spec(R)$.
\end{remark}

Here is the main theorem characterizing when the classical ring of quotients of a ring is zero-dimensional. Condition 5. is new.

\begin{theorem}\cite[Theorem 2.2]{ab}\label{Thm-pi.c.}
Let $R$ be a commutative ring.  The following statements are equivalent.
\begin{enumerate}[label={\rm \arabic*.}, nolistsep]
\item The classical ring of quotients of $R$ is zero-dimensional.
\item For each $x\in R$ there is a $y\in R$ such that $xy\in \fN(R)$ and $x+y\in \reg(R)$.
\item For each $x\in R$ there is a $y\in R$ and a $n\in \N$ such that
$x^ny=0$ and $x^n + y$ is a regular element of $R$.
\item For each $x\in R$ some power of $x$ is complemented.
\item $R$ is almost complemented and $\reg(R)=\areg(R)$.
\end{enumerate}
\end{theorem}

\begin{proof}
Proof of the equivalence of 1. through 3. is in the cited reference. That 3. and 4. are equivalent is definitional.
\vspace{.1in}

2. $\Rightarrow$ 5. Clearly, 2. implies that $R$ is almost complemented. We show that every aregular element is regular. To that end, let $x\in \areg(R)$. It follows that $x\in R\ssm \fN(R)$ and so by 4., there is an $n\in\N$ and $a\in R$ such that $x^na=0$ and $x^n+a\in \reg(R)$. Since $x\in \areg(R)$, $a\in \fN(R)$; choose $j\in \N$ such that $a^j=0$. Then by the Binomial Theorem together with $x^na=0$ produce the equation
$$x^{jn}=x^{nj}+a^{nj}=(x^n+a)^j.$$
It follows that $x^{jn}\in \reg(R)$, and hence also $x\in \reg(R)$.
\vspace{.1in}

5. $\Rightarrow$ 2. We assume that $R$ is almost complemented and $\reg(R)=\areg(R)$. Let $x\in R$. If $x\in \fN(R)$, then setting $y=1$, we get that $xy\in\fN(R)$ and $x+y\in U(R)$. So without loss of generality, we assume that $x\in R\ssm \fN(R)$. Applying Lemma \ref{lem-a.c.}, we can choose $y\in R$ such that $xy\in\fN(R)$ and $x+y\in \areg(R)$. But then $x+y\in\reg(R)$.
\end{proof}

\begin{definition}
Similar to the concept of $\pi$-regular, we say an element $x\in R$ is {\bf $\pi$-complemented} if some power of $x$ is complemented. The ring $R$ is called $\pi$-complemented if every element is $\pi$-complemented. Therefore, $q(R)$ is zero-dimensional if and only if $R$ is $\pi$-complemented if and only if $q(R)$ is $\pi$-complemented.

We show that $a\in R$ is $\pi$-complemented if and only if there is some $b\in R$ such that $ab\in \fN(R)$ and $a+b\in \reg(R)$.
\end{definition}

\begin{lemma}
The element $a\in R$ is $\pi$-complemented if and only if there is some $b\in R$ such that $ab\in\fN(R)$ and $a+b\in \reg(R)$
\end{lemma}

\begin{proof}
Suppose $a^n$ is complemented and let $b\in R$ such that $a^nb=0$ and $a^n+b\in\reg(R)$. By (5) of Lemma \ref{MainLemma}, $a^n$ and $b^n$ are complemented. It follows that $a^nb^n=0$ and $a^n+b^n\in\reg(R)$. So, $ab\in \fN(R)$. Since $(a+b)^n=a^n+b^n+x$ with $x\in \fN(R)$, it follows from (2) of Lemma \ref{MainLemma} that $(a+b)^n$ is regular and therefore $a+b\in\reg(R)$.
\vspace{.1in}

Conversely, suppose that there is some $b\in R$ such that $ab\in\fN(R)$ and $a+b\in \reg(R)$. For some $n\in N$, $(ab)^n=0$. Consider the regular element $(a+b)^n = a^n+b^n+x$ for some nilpotent element $x$. It follows that $a^n+b^n\in\reg(R)$ by (2) of \ref{MainLemma}. Therefore, $a^n$ is complemented, and thus, $a$ is $\pi$-complemented.
\end{proof}

The following two corollaries to Theorem \ref{Thm-pi.c.} ought to be evident.

\begin{corollary}
1) If $R$ is semi-complemented, then $R$ is $\pi$-complemented.

2) If $R$ is $\pi$-complemented, then $R$ is almost complemented.
\end{corollary}


In comparison to \ref{MainLemma}, we have the following.
\begin{lemma}
1) Suppose that $a\in R, s\in\areg(R)$, and $sa$ is an almost complemented element of $R$. Then $a$ is an almost complemented element of $R$.

2) Then $a\in \reg(R)$ if and only if $a\in areg(R)$ and $a$ is complemented.
\end{lemma}

\begin{proof}
1) Since $sa$ is almost complemented it follows that $\overline{sa}$ is a complemented element of $\overline{R}$. Therefore, using that $\overline{s}\in\reg(\overline{R})$, and applying \ref{MainLemma}, $\overline{a}$ is a complemented element of $\overline{R}$. It follows that $a$ is an almost complemented element of $R$.
\vspace{.1in}

2) If $a$ is a complemented aregular element, then for some $b\in R$, $a+b\in\reg(R)$ and $ab=0$. The first tells us that $b\in\fN(R)$ and then we get (Lemma \ref{MainLemma}) that the difference of a regular and a nilpotent is regular, i.e. $a= (a+b) -b \in\reg(R)$.
\end{proof}

\begin{proposition}
The ring $R$ is almost complemented if and only if $q(R)$ is almost complemented.
\end{proposition}

\begin{proof}
Suppose $R$ is almost complemented and let $\frac{x}{s}\in q(R)$. Then there is some $y\in R$ such that $xy\in\fN(R)$ and $x+y\in\areg(R)$. Clearly, $\frac{x}{s}\cdot \frac{y}{s}\in\fN(q(R))$. We claim that $\frac{x+y}{s}\in\areg(q(R))$. To that end, suppose $\frac{c}{t}\cdot \frac{x+y}{s}\in \fN(q(R))$. Then $c(x+y)\in\fN(R)$ and so $c\in\fN(R)$. It follows that $\frac{c}{t}\in\fN(q(R))$. Thus, $q(R)$ is almost complemented.
\vspace{.1in}

Conversely, suppose that $q(R)$ is almost complemented and let $x\in R\ssm\fN(R)$. Then there is some $\frac{y}{s}$ such that $\frac{x}{1}\cdot\frac{y}{s}\in\fN(q(R))$ and $\frac{sx+y}{s}\in\areg(q(R))$. So, $(sx)y\in\fN(R)$ and $sx+y\in\areg(R)$. This shows that $sx$ is almost complemented element of $R$. Since $s$ is regular, the previous lemma ensures us that $x$ is an almost complemented element of $R$.
\end{proof}

\begin{remark}
In \ref{main-example}, we provide an example of an almost complemented ring for which $\reg(R)\subset \areg(R)$. In fact, any example of an almost complemented ring that is not $\pi$-complemented will work.
\end{remark}

\begin{definition}
It is obvious that if $R\ssm \fN(R)=U(R)$, then $R$ is zero-dimensional. These are the precisely the local rings with nil Jacobson radical. Generalizing this slightly, we see that if $R\ssm\fN(R)=\reg(R)$, then $R$ is semi-complemented. These two classes of rings were first mentioned and studied in \cite{snapper}. There, the author called rings for which $R\ssm\fN(R)=\reg(R)$ {\bf primary} rings and rings that satisfy $R\ssm \fN(R)=U(R)$ {\bf completely primary}. These rings, especially the finite ones, have played an important role in the classification of those abelian groups which are the group of units of a ring.

The above terminology has been used even recently, but we should point out that other authors have used these names to mean different things. For example, Lam \cite{lam}, defines a completely primary ring as a local ring whose Jacobson radical is nilpotent. This is stronger than the above definition. We shall avoid the use of these terms altogether, and so for the lack of a better term, we shall say that a ring $R$ satisfies {\bf Property $D$} if $R\ssm \fN(R)=\reg(R)$. And if $R\ssm \fN(R)=\areg(R)$, then we shall say that $R$ satisfies {\bf Property $D^\flat$}. Clearly, $R$ satisfies Property $D^\flat$ if and only if $R/\fN(R)$ is an integral domain. We shall say that $R$ has a unique prime ideal when $R=U(R)\cup \fN(R)$.

It should be clear that if $R$ satisfies Property $D$, then $R$ is semi-complemented. And if $R$ satisfies Property $D^\flat$, then $R$ is almost complemented.

Finally, notice that the ring $R$ satisfies Property $D$ if and only if $q(R)$ satisfies Property $D$ if and only if $q(R)$ has a unique prime ideal.

\end{definition}

\begin{example}
Let $R=F[x_1,x_2,\ldots]/(x_1^2, x_2^3, \ldots, x_n^{n+1},\ldots)$. Then $R$ has a unique prime ideal but its Jacobson radical is not nilpotent.
\end{example}

\begin{remark}
In \cite[Proposition 4.4]{bmz2}, the authors show (in the case of commutative rings and using our notation) that $R\ssm\reg(R)=\fN(R)$ if and only if $R[t, t^{-1}] \subseteq  \fJ(R[t, t^{-1}])$.
\end{remark}

\begin{example}
We consider Example 7 of \cite{km}. Let $K$ be a field and set
$$R = K[x, y]_{(x,y)}/(xy, y^2).$$
We observe that $R$ does not satisfy Property $D$ since $x+(xy,y^2)$ is a zero-divisor that is not nilpotent. It does satisfy Property $D^\flat$ since the nilradical of $R$ is $\fN(R) = (y+ (xy, y^2))$ and $R/\fN(R)$ is isomorphic to $K[x]_{(x)}$, which is a domain.
\end{example}

Here is our main theorem of this section; we characterize the semi-complemented rings. We provide a different proof later on, but mention that the following one lends itself more to a possible generalization to non-commutative rings.

\begin{theorem}\label{main}
$R$ is semi-complemented if and only if $R$ satisfies Property $D$ or $R$ is complemented.
\end{theorem}

\begin{proof}
The sufficiency is clear.
\vspace{.1in}

Let $R$ be a semi-complemented ring. We assume that $R$ does not satisfy Property $D$ and show that $R$ is reduced.
Let $n\in \fN(R)$.

Since $R$  does not satisfy Property $D$, $R$ contains a zero-divisor that is not a nilpotent, say $x$. Let $0\neq y$ be a complement of $x$: $xy=0$ and $x+y\in\reg(R)$. Observe that by Lemma \ref{MainLemma}, $y\notin \fN(R)$. Since $yn\in \fN(R)$ it follows that $x+yn\notin \fN(R)$. So, by hypothesis, $x+yn$ is complemented. Applying (6) of Lemma \ref{MainLemma}, we conclude that $yn=0$. But applying this argument again we conclude that $xn=0$ and so by regularity of $x+y$, $n=0$.
\end{proof}

\vspace{.2in}

Next, we share an illustration of the classes defined above.
\[
\begin{array}
[c]{ccccccc}
     \text{vNR} & \Rightarrow & \text{Compl.}  \\
     \Downarrow && \Downarrow \\
     \text{ZD} & & \text{Semi-Compl.} & \Rightarrow & \text{$\pi$-Compl.} & \Rightarrow & \text{Almost-Compl.} \\
    \Uparrow & & \Uparrow & & & &\Uparrow\\
 \text{Completely Primary}    & \Rightarrow & \text{Property $D$}  && \Rightarrow && \text{Property $D^\flat$}
\end{array}
\]

\vspace{.3in}

\section{Rings of Quotients and Products}

We briefly turn to rings of quotients. Since $\areg(R)$ is a multiplicative system we can form its ring of quotients $\areg(R)^{-1}R$. For the sake of simplicity, we shall denote it by $q_a(R)$; the subscript reminding us that we are taking denominators belonging to the set of aregular elements. We let $\varphi:R\ra q_a(R)$ be the ring homomorphism defined by $\varphi(x)=\frac{x}{1}$.

\begin{lemma}
The ring homomorphism $\varphi:R\ra q_a(R)$ has as its kernel
$$\ker \varphi = \bigcup_{s\in \areg(R)} \Ann(s).$$
Therefore, $\varphi$ is injective if and only if $\reg(R)=\areg(R)$.

Moreover, $\areg(R/\ker\varphi)=\reg(R/\ker\varphi)$, and furthermore, $q(R/\ker\varphi)=q_a(R)$.

\end{lemma}

\begin{proof}
Note that $\frac{x}{1}=\frac{0}{1}$ if and only if there is some $s\in \areg(R)$ such that $xs=0$. Thus, the characterization of the kernel is complete. The second statement follows from the first.
\vspace{.1in}

Let $\eta(R)=\ker\varphi$ and $S=R/\eta(R)$. Let $x\in R$ satisfy $x+\eta(R)\in\areg(S)$. We aim to show that $x\in \areg(R)$. To that end, let $y\in R$ such that $xy\in\fN(R)$. Choose $n\in \N$ such that $x^ny^n=0$. It follows that $xy+\eta(R)\in \fN(S)$ and so, by choice of $x$, we know that $y+\eta(R)\in \fN(S)$. This implies that $(y+\eta(R))^t=0$ for some natural $t$, i.e. $y^t\in \eta(R)$. In turn, this implies that there is some $z\in \areg(R)$ such that $zy^t=0$. It follows that $zy\in \fN(R)$ and since $z$ is aregular, $y\in \fN(R)$. Recapping, we showed that $xy\in \fN(R)$ implies $y\in \fN(R)$. Therefore, $x\in\areg(R)$.
\vspace{.1in}

Next, we show that if $x\in \areg(R)$, then $x+\eta(R)\in \reg(S)$. To that end, if $(x+\eta(R))(y+\eta(R))=0$, then $xy\in \eta(R)$ which implies that for some $s\in \areg(R)$, $(xs)y=(xy)s=0$. Since $xs\in\areg(R)$, $y\in \eta(R)$, whence $y+\eta(R)=0$. It follows that $x+\eta(R)\in\reg(R/\eta(R))$.
\vspace{.1in}

Putting these two pieces together yields that $\areg(R/\eta(R))\subseteq \reg(R/\eta(R))$. The reverse containment is always true, whence we have equality. The last statement follows from the well-known condition that the localization at a multiplicative set can also be obtained by first factoring out its kernel, and then localizing at the factor's multiplicative set of regular elements.
\end{proof}

\vspace{.3in}
We observe that due to the definition of aregular elements, an element annihilated by an aregular is nilpotent. Therefore, the ideal $\eta(R)$ is contained in the nilradical of $R$.

\begin{definition}
We say the ring $R$ is {\bf nearly reduced} if $\eta(R)=\{0\}$. Of course, this is equivalent to saying that $\reg(R)=\areg(R)$.
\end{definition}

Obviously, a reduced ring is nearly reduced. Any $\Z/p^k\Z$ is nearly reduced and not reduced (as long as $k\geq 2$). Furthermore, by Theorem \ref{Thm-pi.c.}, an almost complemented ring is $\pi$-complemented if and only if it is nearly reduced.

\vspace{.3in}

We next consider products. Recall that $R=\displaystyle{\Pi_{i\in I} R_i}$ is von Neumann regular (resp. complemented) if and only if each $R_i$ is von Neumann regular (resp. complemented). In the finite case, say $R=S\times T$, since $\fN(R)=\fN(S)\times \fN(T)$ and $R/\fN(R)=S/\fN(S)\times T/\fN(T)$, we also gather that $R$ is almost complemented if and only if both $S$ and $T$ are both almost complemented. Furthermore, $\reg(R)=\reg(S)\times \reg(T)$, and so it also follows that $\areg(R)=\areg(S)\times \areg(T)$. In the infinite case, once must take care to determine when a product is zero-dimensional. Recall that a collection of rings $\{R_i\}$ is said to be of bounded nilpotence if there is some $k\in\N$ such that for all $i\in I$ and all $x\in \fN(R_i)$, $x^k=0$. For example, the direct product of an infinite family of copies of $\Z/4\Z$ is of bounded nilpotence, whilst the product of the family $\{\Z/2^k\Z\}_{k\in\N}$ is not. The following is known; we include a proof for completeness sake. See the reference for more equivalent statements.

\begin{remark}
When studying products $\Pi_{i\in I} R_i$ we let $\chi_i$ denote the function that sends $i$ to 1 and every other index to 0. For an $a_i\in R_i$ we let $a_i\chi_i$ denote the function that sends $i$ to $a_i$ and every other index to 0.
\end{remark}

\begin{proposition}\cite[Theorem 3.4]{gh}
Let $J=\{R_i\}_{i\in I}$ be a family of rings and let $R=\Pi_{i\in I} R_i$. The following statements are equivalent.
\begin{enumerate}[label={\rm \arabic*.}, nolistsep]
\item $R$ is zero-dimensional.
\item Each $R_i$ is zero-dimensional and $\fN(R)=\Pi_{i\in I} \fN(R_i)$.
\item Each $R_i$ is zero-dimensional and $J$ is of bounded nilpotence.
\end{enumerate}
\end{proposition}

\begin{proof}
Clearly, 2. and 3. are equivalent. Suppose that 1. is true. Since the homomorphic image of a zero-dimensional ring is zero-dimensional, if $R$ is zero-dimensional, then so is each $R_i$. Let $f=(f_i)\in \fN(R)$, i.e.  each $f_i\in \fN(R_i)$. Since $R$ is zero-dimensional some power of $f$ is von Neumann regular, say $f^n$. Then in $R_i$, $f_i^n$ is von Neumann regular and nilpotent, so $f_i^n=0$. Hence $f^n=0$.
Thus, $f\in \Pi_{i\in I} \fN(R_i)$ and so 2. holds.
\vspace{.1in}

Next, suppose that 3. is true. In order to show that $R$ is zero-dimensional we take $s=(r_i)\in R$ and show that some power of $s$ is von Neumann regular. In \cite[Theorem 3.1]{huckaba}, it is shown in the proof of (3) $\Rightarrow$ (2), that if $r,x\in R_i$ and $n\in \N$ satisfies
$$0=(r-r^2x)^n,$$
then there is some $y\in R_i$ such that $r^{n+1}y=r^n$. Next, in the proof of (2) $\Rightarrow$ (1) of \cite[Theorem 3.1]{huckaba}, it is shown that the previous equality implies that $r^n=r^{2n}y^n$.   By 3., there is an $n\in\N$ such that for all $i\in I$, $0=(r_i-r_i^2x_i)^n$. And so for each $i$ there is a $y_i\in R_i$ such that $r_i^n=r_i^{2n}y_i^n$. Let $x=(y_i)\in R$. Then $s^n=s^{2n}x$, whence $s$ is $\pi$-regular, and therefore $R$ is zero-dimensional.
\end{proof}

We turn to the $\pi$-complemented condition for arbitrary products. To our knowledge this result is new.

\begin{proposition}
Let $\{R_i\}_{i\in I}$ be a family of rings and let $R=\Pi_{i\in I} R_i$. The following statements are equivalent.
\begin{enumerate}[label={\rm \arabic*.}, nolistsep]
\item $R$ is $\pi$-complemented.
\item Each $R_i$ is $\pi$-complemented and $\fN(R)=\Pi_{i\in I} \fN(R_i)$.
\item Each $R_i$ is $\pi$-complemented and the family is of bounded nilpotence.
\end{enumerate}
\end{proposition}

\begin{proof}
That 2. and 3. are equivalent is clear.
\vspace{.1in}

1. $\Rightarrow$ 2. Suppose that $R$ is $\pi$-complemented and let $i\in I$ and $a_i\in R_i$. Let $f=a_i\chi_i$. By hypothesis, there is some $n\in \N$ such that $f^n$ is complemented. Let $g=(g_i)\in R$ be a complement of $f^n$. Then $a_i^ng_i=0$ and $a_i^n+g_i\in\reg(R_i)$. It follows that $a_i$ is a $\pi$-complemented element of $R_i$. Consequently, each $R_i$ is $\pi$-complemented.

Next, let $f=(f_i)\in \Pi_{i\in I} \fN(R_i)$ and choose $n\in \N$ such that $f^n$ is complemented. Let $g=(g_i)\in R$ be a complement of $f^n$. Since, for each $i\in I$, $f_i^n$ is complemented and nilpotent it follows by Lemma \ref{MainLemma} that $f_i^n=0$ for each $i$; whence $f^n=0$. Therefore, $\fN(R)=\Pi_{i\in I} \fN(R_i)$.
\vspace{.1in}

2. $\Rightarrow$ 1. Let $f=(f_i)\in R$. We use condition 2. of Theorem \ref{Thm-pi.c.}. For each $i\in I$ there is a $y_i\in R_i$ such that $f_iy_i\in\fN(R)$ and $f_i+y_i\in\reg(R_i)$. Let $g=(y_i)\in R$. Then $fg\in \Pi_{i\in I} \fN(R_i)=\fN(R)$ and $f+g\in\reg(R)$. It follows that $R$ is $\pi$-complemented.
\end{proof}

Next, we attempt to address when the product of a family of rings is almost complemented; unfortunately we do not have a complete answer. Recall that an element is a regular element of the product if and only if each coordinate is a regular element of its factor. Interestingly, the same cannot be said for aregular elements. However, we do get one direction.

\begin{lemma}\label{areg}
Let $\{R_i\}_{i\in I}$ be a family of rings and let $R=\Pi_{i\in I} R_i$. If $r=(r_i)\in \areg(R)$, then $r_i\in \areg(R_i)$ for each $i\in I$.

\end{lemma}

\begin{proof}
First suppose that $r_i\in\fN(R_i)$. Observe that $\chi_i\notin\fN(R)$ but $r_i\chi_i\in \fN(R)$, contradicting our choice of $r\in\areg(R)$. Therefore, $r_i\notin\fN(R)$.

Next, suppose that $x_i\in R_i$ and $r_ix_i\in \fN(R_i)$. Let $f=x_i\chi_i$. Then $rf\in \fN(R)$, whence $f\in \fN(R)$ from which we gather that $x_i\in\fN(R_i)$.
\end{proof}

Here are a couple of consequences of this result.

\begin{proposition}
Suppose $R=\Pi_{i\in I} R_i$. Then $R$ is nearly reduced if and only if each $R_i$ is nearly reduced.
\end{proposition}

\begin{proof}
Suppose that $R$ is nearly reduced and let $j\in I$ and $a_j\in \areg(R_j)$. Let $f\in R$ be defined by
$$f(i)=\begin{cases}
       a_j , & \mbox{if } i=j \\
       1 , & \mbox{otherwise}.
      \end{cases}$$

We claim that $f\in \areg(R)$. Clearly, $f\notin \fN(R)$. So suppose $g=(g_i)\in R$ and that $fg\in \fN(R)$. It follows that for each $i\neq j$, $g_i\in \fN(R_i)$ and $a_jg_j\in \fN(R_j)$. Since $a_j\in\areg(R_j)$, then $g_j\in \fN(R_j)$. Choose $n\in \N$ such that both $(fg)^n=0$ and $g_j^n=0$. Then $g^n=0$. Thus, $f\in \areg(R)$. By hypothesis, $f\in\reg(R)$ and so $a_j\in \reg(R_j)$.
\vspace{.1in}

Next suppose that each $R_i$ is nearly reduced and let $f=(f_i)\in\areg(R)$. Then, by Lemma \ref{areg}, each $f_i\in\areg(R_i)=\reg(R_i)$. It follows that $f\in\reg(R)$.
\end{proof}

\begin{lemma}\label{nia}
Let $\{R_i\}_{i\in I}$ be a family of rings and let $R=\Pi_{i\in I} R_i$. If $\fN(R)=\Pi_{i\in I} \fN(R_i)$, then $\areg(R)=\Pi_{i\in I} \areg(R_i)$.
\end{lemma}

\begin{proof}
Let $g=(g_i)\in \Pi_{i\in I} \areg(R_i)$. Then each $g_i\in\areg(R_i)$ and so $g_i\notin\fN(R_i)$ for each $i\in I$. It follows that $g\notin\fN(R)$. Next, suppose that $fg\in\fN(R)$ for $f=(f_i)\in R$. Since for each $i$, $g_if_i\in \fN(R_i)$ it follows that $f_i\in \fN(R_i)$ and so $f\in \Pi_{i\in I} \fN(R_i)$. Whence $f\in \fN(R)$.
\end{proof}

So far our next two results are the best we can say about when a product is almost complemented. We have been unable to show that a family of rings whose product is almost complemented must have bounded nilpotence.

\begin{proposition}
Let $\{R_i\}_{i\in I}$ be a family of rings and let $R=\Pi_{i\in I} R_i$.  If $R$ is almost complemented, then each $R_i$ is almost complemented. With the added assumption that the family has bounded nilpotence, the converse holds.
\end{proposition}

\begin{proof}
Suppose that $R$ is almost complemented and let $a_i\in R_i$. Let $g=(g_i)$ be an almost complement for the element $a_i\chi_i$. Then notice that $a_ig_i\in\fN(R_i)$ and $a_i+g_i\in\areg(R_i)$. It follows that $R_i$ is almost complemented.

For the reverse direction, assume that $\fN(R)=\Pi_{i\in I} \fN(R_i)$ and let $f=(f_i)\in R$. For each $i\in I$, let $g_i\in R_i$ be an almost complement of $f_i$ and set $g=(g_i)$. Then since for each $i\in I$ $f_ig_i\in\fN(R_i)$, the hypothesis yields that $fg\in \fN(R)$. Similarly, by Lemma \ref{nia}, $f+g\in\areg(R)$.
\end{proof}

\begin{proposition}
Let $\{R_i\}_{i\in I}$ be a family of rings and let $R=\Pi_{i\in I} R_i$. Then the following statements are equivalent.
\begin{enumerate}[label={\rm \arabic*.}, nolistsep]
\item $R$ is almost complemented and $\fN(R)=\Pi_{i\in I} \fN(R_i)$.
\item $R$ is almost complemented and $\areg(R)=\Pi_{i\in I} \areg(R_i)$.
\end{enumerate}

\end{proposition}

\begin{proof}
1. $\Rightarrow$ 2. Clear by Lemma \ref{nia}.
\vspace{.1in}

2. $\Rightarrow$ 1. Let $f=(f_i)\in \Pi_{i\in I} \fN(R_i)$. Let $g=(g_i)\in R$ satisfy $fg\in \fN(R)$ and $f+g\in \areg(R)$. Since each $f_i\in \fN(R_i)$ and $f_i+g_i\in\areg(R_i)$ it follows that each $g_i\in\areg(R_i)$. By hypothesis, $g\in \areg(R)$. Consequently, $f\in \fN(R)$.
\end{proof}

\begin{example}
Observe that if $\{R_i\}_{i\in I}$ is a collection of $\pi$-complemented rings, then since each $R_i$ is nearly reduced it follows that the product is nearly reduced, and therefore the product is almost complemented if and only if the product is $\pi$-complemented. This also shows that if each $R_i$ is almost complemented and $\areg(R)=\Pi_{i\in I} \areg(R)$, then it does not follow that the family has bounded nilpotence.
\end{example}
\vspace{.3in}

It should be apparent that if one of the factors of a product has non-zero nilpotent elements, then the product cannot be semi-complemented. In other words, if $R$ is semi-complemented and not reduced, then $R$ is indecomposable.
We can now supply a different proof of our main theorem from the last section: Theorem \ref{main}.
\begin{theorem}
The ring $R$ is semi-complemented if and only if $R$ satisfies Property $D$ or $R$ is complemented.
\end{theorem}

\begin{proof}
Suppose $R$ is semi-complemented and further assume that $R$ is not reduced. Then $q(R)$ is a semi-complemented zero-dimensional ring that is not reduced; hence $q(R)$ is indecomposable. Take $x\in q(R)\ssm \fN(q(R))$. Then some power of $x$ is von Neumann regular, say $x^n$. This implies that $x^nR=eR$ for some idempotent $e\in R$. Since $q(R)$ cannot have any direct summands it follows that either $e=0$ or $e=1$, but it cannot be the former since $x$ is not nilpotent. It follows that $x^n$ and hence $x$ is a unit. Consequently, $q(R)\ssm\fN(q(R))=U(q(R))$. Consequently, $R$ satisfies Property $D$.
\end{proof}

\vspace{.2in}

\section{Constructing Semi-Complemented Rings}
We would like to share a construction that can be used to produce other semi-complemented rings. Throughout this section, $T$ denotes a ring with a unital subring $S$ and an ideal $I$ such that $T=S+I$ and $S\cap I=\{0\}$. The latter two conditions imply that every element of $T$ has a unique representation $t=s+i$ with $s\in S$ and $i\in I$. Many of the constructions of commutative rings can be viewed this way. Here are some examples.

\begin{example}
The quintessential example of the construction $T=S+I$ is the idealization of a ring $R$ over an $R$-algebra $V$: $T=I(R,V)$, $S=R$, and $I=V$. As a set $I(R,V)=R\times V$ with coordinate-wise addition and multiplication defined by $(r,v)(s,u)=(rs,ru+sv+uv)$. In fact, the idealization of a ring is the external way of creating the $T=S+I$ construction.

1) A special case of the idealization is the trivial extension where we have an $R$-module $M$ equipped with the trivial multiplication: for all $m,n\in M$, $mn=0$. This construction is denoted by $R\propto M$. (We point out that we assume the left and right $R$-module structure is the same: for all $r\in R$, $m\in M$, $rm=mr$.)
\vspace{.1in}

2) The $A+B$ construction: $T=A+B$, $S=A$, and $I=B$.
\vspace{.1in}

3) The polynomial ring: $T=R[\{x_i\}]$, $S=R$, and $I=\langle \{x_i\}\rangle$.
\vspace{.1in}

\end{example}

Notice that in $T=S+I$, since $I$ is an ideal it is an $S$-module. Recall that an $R$-module $M$ is said to be {\bf torsion-free} if for all $r\in \reg(R)$ and $m\in M$, $rm=0$ implies $m=0$.

The proof of the following should be evident.
\begin{lemma}\label{tf}
Let $T=S+I$. The ideal $I$ is a torsion-free $S$-module if and only if $\reg(T)\cap S=\reg(S)$.
\end{lemma}
\vspace{.2in}

\subsection{Nil Ideals}

We begin by considering the case that $I$ is a nil ideal.
\begin{lemma}\label{aregnil}
Suppose $T=S+I$ and $I$ is a nil ideal. Then $\areg(T)\cap S=\areg(S)$.
\end{lemma}

\begin{proof}
Let $s\in \areg(T)\cap S$. Then $s\notin \fN(S)$. Let $r\in S$ satisfy $rs\in \fN(S)$. Then $rs\in\fN(T)$ and so $r\in \fN(T)$ and thus also $r\in \fN(S)$. It follows that $s\in \areg(S)$.
\vspace{.1in}

Conversely, suppose $s\in\areg(S)$. Then $s\notin \fN(T)$. Let $r+j\in T$ satisfy $s(r+j)\in\fN(T)$. Since $j$ is nilpotent it follows that $sr\in\fN(T)$ and so $sr\in\fN(S)$, thus $r\in \fN(T)$ and also $r\in \fN(S)$.
\end{proof}

\begin{theorem}\label{fishD}
Let $T=S+I$ and suppose that $I$ is a non-zero nil ideal of $T$. Then the following are equivalent.
\begin{enumerate}[label={\rm \arabic*.}, nolistsep]
\item $T$ is semi-complemented.
\item $T$ has Property $D$.
\item $S$ has Property $D$ and $I$ is a torsion-free $S$-module.
\end{enumerate}
\end{theorem}

\begin{proof}
1. $\Rightarrow$ 3. Suppose $T$ is semi-complemented. Then since $T$ has non-zero nilpotent elements it follows from Theorem \ref{main} $T$ has Property $D$. Let $s\in S\ssm \fN(S)$. Then $s\in T\ssm \fN(T)$ and so $s\in \reg(T)\cap S\subseteq \reg(S)$. Therefore, $S$ has Property $D$ and $I$ is a torsion-free $S$-module.
\vspace{.1in}

3. $\Rightarrow$ 2. Let $t=s+i\in T\ssm \fN(T)$. Now, $s$ is either a regular element of $S$ and hence of $T$, or it is nilpotent. In the first case, $t$, being the sum of a regular and nilpotent, is regular. In the second case $t$ is nilpotent. Therefore, $T$ has Property $D$ and hence is semi-complemented.
\vspace{.1in}

2. $\Rightarrow$ 1. Obvious
\end{proof}

\begin{corollary}\label{CorfishD}
1) For a ring $R$ and an $R$-module $M$, $R\propto M$ is semi-complemented if and only if $R$ satisfies Property $D$ and $M$ is a torsion-free $R$-module.
\vspace{.1in}

2) Let $R$ be a ring and set $T=R[x]/(x^{n+1})$ with $n\geq 1$. Then $T$ is semi-complemented if and only if $R$ satisfies Property $D$.
\vspace{.1in}

3) Consider $Q=R[\{x_i\}]$ the polynomial ring over $R$ in a countable number of indeterminates. Let $J$ be the ideal generated by the set $\{x_n^{n+1}\}$ and set $T=Q/J$. Then the ideal $I$ of $T$ generated by the set $\{x_n\}$ is a nil ideal and a torsion-free $R$-module; also $T=R+I$. Therefore, $T$ satisfies Property $D$ if and only if $R$ satisfies Property $D$.
\vspace{.1in}

\end{corollary}

We continue with the case that $I$ is nil and classify when $T=S+I$ is either almost complemented or $\pi$-complemented. We leave the proof of the following lemma to the interested reader.

\begin{lemma}
Suppose that $I$ is nil and $T=S+I$, and let $t=s+i\in T$. The following are equivalent.
\begin{enumerate}[label={\rm \arabic*.}, nolistsep]
\item $t\in\areg(T)$.
\item $s\in\areg(T)$
\item $s\in\areg(S)$
\end{enumerate}
\end{lemma}

\begin{proposition}
Let $T=S+I$ and $I$ a nil ideal.

1) $T$ is almost complemented if and only if $S$ is almost complemented.
\vspace{.1in}

2) $T$ is $\pi$-complemented if and only if $S$ is $\pi$-complemented and $I$ is a torsion-free $S$-module.
\vspace{.1in}

3) $T$ has Property $D^\flat$ if and only if $S$ has Property $D^\flat$.
\end{proposition}

\begin{proof}
For the proof of 1) and 3) observe that since $I$ is nil, $T/\fN(T)=S/\fN(S)$.

For 2), use 1), Lemma \ref{aregnil} and Theorem \ref{Thm-pi.c.}.
\end{proof}

\begin{example}\label{main-example}
Consider $R=\Z\propto \Z_5$. This ring satisfies Property $D^\flat$, and hence is an almost complemented ring. The element $(5^k,n)$ is not regular for any $n\in \Z_5$ and any $k\in\N$: $(5^k,n)(0,1)=(0,0)$. Furthermore, the element $(5^k,n)$ is also not complemented. To see this first observe that if $(0,0)=(5^k,n)(z,m)=(5^kz,zn)$, then $z=0$. Secondly, $(5^k,n)+(0,m)=(5^k,n+m)$ which is not regular. Consequently, $R$ is not $\pi$-complemented.
\end{example}
\vspace{.1in}

\begin{example}\label{ex2}
Observe that if $R$ is any complemented ring that is not an integral domain, then $R\propto R$ is a $\pi$-complemented ring that is not semi-complemented. Moreover, if $R$ is a von Neumann regular ring that is not a field, then for any non-trivial ideal $I$, $R\propto R/I$ is a $\pi$-complemented ring that is not semi-complemented.

\end{example}

\begin{remark}
There are many examples of torsion-free $R$-modules: any ring extension $R\leq S\leq Q(R)$ where $Q(R)$ is the maximal ring of quotients of $R$, as well as any product of such rings. The polynomial rings and power series rings over $R$ are also torsion-free $R$-modules. Any non-zero ideal of $R$ is torsion-free. On the other hand the module $M=R/I$ for some proper ideal $I$ need not be torsion-free. For $M$ to be torsion-free, $I$ must satisfy the property that if $x\notin I$, then $rx\notin I$ for all $r\in \reg(R)$. In particular, it is necessary that $I$ cannot contain any regular elements. However, this is not sufficient. For example, suppose $R$ is complemented and $I$ is not a radical ideal and contains no regular elements. Let $x\in R$ satisfy $x\notin I$ but $x^2\in I$. Then for any complement $y$ of $x$, $(x+y)x\in I$.
\vspace{.1in}

For a specific example, let $R=C(\R)$ and let $f\in C(\R)$ be the function defined by
$$f(x) =\begin{cases}
    x & \mbox{if } x\geq 0 \\
    0, & \mbox{otherwise}.
  \end{cases}$$
Then $I=f^2C(\R)$ does not contain any regular elements yet $C(\R)/I$ is not torsion-free.
\vspace{.1in}

It follows that if $R$ is semi-complemented and $R/I$ is torsion-free, then either $I$ is a radical ideal or $I\subseteq \fN(R)$. When $R$ is reduced we get a partial converse. If $I$ is an intersection of minimal prime ideals, say $I=\cap P_i$, then $R/I$ is torsion-free, since if $rx\in I$ and $r\in \reg(R)$, then because $r\notin P_i$ for each $i$, we gather that $x\in I$. For example, $R/\fN(R)$ is a torsion-free $R$-module.

We now point out that when $R\propto R/I$ is semi-complemented, then it is necessary that $I$ be a nil ideal. Example \ref{ex2} shows that we cannot extend this to $\pi$-complemented rings.
\end{remark}

\begin{proposition}
Let $T=R\propto R/I$, where $I$ is a proper ideal of $R$. The following are equivalent:
\begin{enumerate}
\item $T$ is semi-complemented.
\item $T$ is semi-complemented and $I\subseteq {\rm nil}(R)$.
\item $T$ satisfies Property $D$.
\item $R$ satisfies Property $D$ and $R/I$ is torsion-free.
\end{enumerate}
\end{proposition}

\begin{corollary}
The ring $R\propto R/\fN(R)$ is semi-complemented if and only if $R$ satisfies Property $D$.
\end{corollary}

\begin{example}
Here is an example of a ring $R$ satisfying Property $D$ with an ideal $I\subset \fN(R)$ for which $R/I$ is not torsion-free. It follows that $R\propto R/I$ is not semi-complemented.

Let $R=\Z\propto \Z$ and let $I=(0,2)R=0\propto 2\Z$. Then $(0,1)\notin I$ but $(2,0)(0,1)\in I$. This argument generalizes to any ring with a non-zero nilpotent element $n\in \fN(R)$ and an $r\in\reg(R)\ssm U(R)$. Then let $I=rnR$. Observe that $n\notin I$, otherwise, there would exist an $s\in R$ such that $n=rsn$ and so in $\overline{R}$, $rs=1$, implying that $r\in U(R)$, a contradiction. Then $R/I$ is not a torsion-free $R$-module. It follows that if every nil ideal $I$ satisfies $R/I$ is torsion-free, then either $R$ is reduced or $\reg(R)=U(R)$. The latter of these two conditions is equivalent to the saying that every $R$-module is torsion-free.
\end{example}

\begin{proposition}
For a ring $R$ the following statements are equivalent.
\begin{enumerate}[label={\rm \arabic*.}, nolistsep]
\item $R$ has a unique prime ideal.
\item $R\ssm \fN(R)=U(R)$.
\item $R\propto M$ is semi-complemented for all $R$-modules $M$.
\end{enumerate}
\end{proposition}
\vspace{.3in}

\subsection{Polynomial and power series rings}

Recall that a polynomial ring $R[x]$ is complemented if and only if $R$ is reduced and has a compact minimal prime spectrum (see \cite[Corollary 4.6]{huckaba}). Internally, the condition that defines the latter is the following: for each $a\in R$ there is a finitely generated ideal $J$ such that $aJ=0$ and $aR+J$ is a dense ideal (i.e. zero annihilator.) The reason that one does not need the full force of $R$ being complemented is that $R[x]$ always has Property $A$. We turn and answer when $R[[x]]$ is complemented as well as address some other issues.

Clearly, $R[x]$ and $R[[x]]$ are both torsion-free $R$-modules. The interesting question which we have been unable to solve is for which $R$ is $R[[x]]$ a torsion-free $R[x]$-module. We do show it for Property $D$.

\begin{lemma}\label{lemmapoly}
If $R$ satisfies Property $D$, then $R[[x]]$ is a torsion-free $R[x]$-module.
\end{lemma}

\begin{proof}
By Lemma \ref{tf} it suffices to show that each regular element of $R[x]$ is a regular element of $R[[x]]$. To that end, let $f(x)\in\reg(R[x])$ and let $f(x)=f_0+f_1x+\cdots + f_nx^n$. Then at least one of the coefficients of $f(x)$ is not nilpotent, and therefore is regular in $R$. By \cite[Lemma 2]{fields}, $f(x)\in \reg(R[[x]])$.
\end{proof}

\begin{proposition}\label{poly}
Let $R$ be a ring. Then $R$ satisfies Property $D$ if and only if $R[x]$ satisfies Property $D$.
\end{proposition}

\begin{proof}
The sufficiency is clear. For the necessity, let $f(x) = \sum_{i=0}^n a_ix^i \in R[x]\ssm \fN(R[x])$. Then there exists $0\leq k\leq n$ such that $a_k \notin \fN(R)$ and $a_i\in \fN(R)$ for $i = 0, \ldots , k-1$. Since $a_k\notin \fN(R)$ it also is not in $\fN(R[x])$ and thus, $a_k\in \reg(R[x])\cap R=\reg(R)$. Therefore, $g(x) = \sum^{k-1}_{i=0} a_ix^i \in \fN(R[x])$, and $h(x) = \sum_{i=k}^n a_ix^i \in \reg(R[x])$. Consequently, $f(x) = g(x) + h(x) \in\reg(R[x])$ by Lemma \ref{MainLemma}.
\end{proof}

\begin{proposition}
Suppose $R$ satisfies Property $D$. Then, $R[x]\propto R[[x]]$ satisfies Property $D$.
\end{proposition}

\begin{proof}
Suppose that $R$ satisfies Property $D$. Then so does $R[x]$ by Proposition \ref{poly}. The rest follows by Lemma \ref{lemmapoly} and  Corollary \ref{CorfishD}.
\end{proof}

At this point we have been unable to characterize when $R[[x]]$ has Property $D$. Clearly, it is necessary that $R$ has Property $D$. But it is not sufficient as we now supply an example.

\begin{example} Let $F$ be a field and set $R = F[x_1, x_2, \ldots]/K$ where $K$ is the ideal generated by the collection $\{x_i^{i+1}:i\in \N\}\cup \{x_ix_j:i\neq j\}$. The following hold.
\begin{enumerate}
\item $R/\fN(R) \cong F$; so $R = \fN(R)\cup U(R)$.
\item $\fN(R)$ is not nilpotent.
\item $R[[t]]$ does not have Property $D$.
\end{enumerate}
\end{example}

\begin{proof}
(1) and (2). These are easily proved.

(3) Let $f(t) = x_1+x_2t+\ldots+x_nt^{n-1}+\ldots \in R[[t]]$. First, $f(t)$ is not nilpotent. Indeed,
if $f(t)^k = 0$ for some $k \geq 2$, then, since $(x_1 +x_2t+\ldots+x_{k-1}t^{k-2})^k = 0$, it follows that
$$(x_kt^{k-1}+x_{k+1}t^{k}+x_{k+2}t^{k+1}+\ldots  )^p=[f(x)-(x_1 +x_2t+\ldots+x_kt^{k-1})]^p = 0$$
for some natural $p\in \N$, a contradiction. So $f(t)\in R[[t]]$ is not nilpotent. But $f(t)$ is a zero-divisor as $x_1f(t)=0$.
\end{proof}

To our knowledge, a classification of when the power series over $R$ is a complemented ring has not been expressed in the literature. We need a way, as in the case for polynomials, to determine when a power series is a zero-divisor. For any power series $f(x)\in R[[x]]$, we let $c(f)$ denote the (countably generated) ideal of $R$ generated by the coefficients of $f(x)$.

\begin{lemma}
Let $R$ be a reduced ring and suppose $f(x)=\sum_{i=0}^\infty f_ix^i\in R[[x]]$. Then $f(x)\in\reg(R[[x]])$ if and only if $c(f)$ is a dense ideal of $R$.
\end{lemma}

\begin{proof}
If $c(f)$ is not a dense ideal, then there is some non-zero $r\in R$ such that $rc(f)=0$. It follows that $rf(x)=0$.
\vspace{.1in}

Conversely, suppose that $f(x)$ is annihilated by $g(x)\in R[[x]]$. For the sake of completeness we show that $c(f)c(g)=0$. (one can find the proof in \cite{brewer}.) Write $g(x)=\sum_{i=0}^\infty g_ix^i$ and write $f(x)g(x)=\sum_{i=0}^\infty c_ix^i$, where $c_n=f_0g_n+f_1g_{n-1}+\cdots + f_{n-1}g_1+ f_ng_0=0$. Observe that $c_0=0$ implies that $f_0g_0=0$. Also, $0=f_0g_1+f_1g_0$ and so multiplying both sides by $f_0$ produces the equation $(f_0g_1)^2=0$ and since $R$ is reduced. We get that $f_0g_1=0$. Similarly, $f_1g_0=0$. Moving to $c_2=f_0g_2+f_1g_1+f_2g_0$, again we first multiply by $f_0$ and observe that $f_0g_2=0$, then multiply by $g_0$ to get that $f_2g_0=0$ and hence $f_1g_1=0$.

Continuing on in this inductive manner shows that $f_ig_j=0$ for all $i,j\in\N$.
\end{proof}

\begin{corollary}
Suppose $R$ is a reduced ring. Then $R[[x]]$ is a torsion-free $R[x]$-module.
\end{corollary}

\begin{proof}
Let $f(x)\in\reg(R[x])$. Then $c(f)$ is a dense ideal of $R$, and so $f(x)\in \reg(R[[x]])$.
\end{proof}

Our next result characterizes when $R[[x]]$ is a complemented ring. The result is not surprising when you recall the appropriate statement for $R[x]$.

\begin{theorem}
The power series ring $R[[x]]$ is a complemented ring if and only if $R$ is reduced and for each countably generated ideal $I$ there is  countably generated ideal $J$ such that $IJ=0$ and $I+J$ is dense.
\end{theorem}

\begin{proof}
Suppose $R[[x]]$ is a complemented ring. Then, $R[[x]]$ and hence $R$ is reduced. Let $I$ be a countably generated ideal, and let $f_0,f_1,\ldots,$ be a countable generating set for $I$ and set $f(x)=\sum_{i=0}^\infty f_ix^i$. Note that $c(f)=I$. By hypothesis, there is some complement of $f(x)$ in $R[[x]]$, say $g(x)=\sum_{i=0}^\infty g_ix^i$. It follows that $g(x)f(x)=0$ and $f(x)+g(x)\in \reg(R[[x]])$. The first part of that implies that $Ic(g)=c(f)c(g)=0$ where $c(g)$ is a countably generated ideal. The second part implies that $c(f+g)$ is a dense ideal. Since $c(f+g)\subseteq c(f)+c(g)$ it follows that $c(f)+c(g)$ is also a dense ideal of $R$.
\vspace{.1in}

Conversely, let $f(x)=\sum_{i=0}^\infty f_ix^i\in R[[x]]$. Choose a countably generated ideal $J$ such that $c(f)\cdot J=0$ and $c(f)+J$ is a dense ideal of $R$. Let $g(x)=\sum_{i=0}^\infty g_ix^i\in R[[x]]$ satisfy $c(g)=J$. Then clearly, $f(x)g(x)=0$. We show that $c(f+g)$ is also dense ideal of $R$. Let $r\in R$ be such that $rc(f+g)=0$. It follows that for all $n\in\N$, $r(f_n+g_n)=0$. Since $c(f)c(g)=0$ and $R$ is reduced we gather that for each $n\in\N$, $rf_n=rg_n=0$ and so $r(c(f)+J)=0$. By density, $r=0$.
\end{proof}

\begin{example}
Let $R$ be a reduced ring for which every ideal is countably generated. Then $R[[x]]$ is complemented since for any countably generated ideal $I$, the ideal $\Ann_R(I)$ is also countably generated and clearly $I\Ann_R(I)=0$ and $I+\Ann_R(I)$ is dense. Notice that this holds in the case that $R$ is countable and reduced.
\end{example}

A natural question is whether $R[[x]]$ being complemented implies that $\Min(R)$ is compact. We show that this is not the case.

\begin{example}
Let $A=\Z$ and take $\cP=\Max(\Z)$. The ring $R=A+B$ is reduced but not complemented since $\cP\nsubseteq \Min(\Z)$. However, $R$ does satisfy Property $A$ \cite[Proposition 2.1]{lucas86}. Therefore, we conclude that $\Min(R)$ is not compact. Notice that $B$ is a countable set since $\cI=\cP\times\N$ is a countable set. Since $A$ is also countable it follows that $R$ is a countable ring. Consequently, $R[[x]]$ is complemented.

\end{example}

\section{Roughly Complemented Elements and Rings}
When studying the difference between complemented elements and almost complemented elements, we see that we have two choices for the product $ab=0$ or $ab\in\fN(R)$, and also two choices for the sum: $a+b\in\reg(R)$ or $a+b\in \areg(R)$. The two extremes are complemented and almost complemented. We also know that the condition there is a $b\in R$ such that $ab\in\fN(R)$ and $a+b\in\reg(R)$ is saying that $a$ is $\pi$-complemented. Thus, we are left with one more variation.

\begin{definition}
We say the element $a\in R$ is {\bf roughly complemented} if there is some $b\in R$ such that $ab=0$ and $a+b\in\areg(R)$. (When this happens, we shall say that $a$ and $b$ are rough complements.) If every element of $R$ is roughly complemented, then we say $R$ is roughly complemented.
\end{definition}

Theorem \ref{Thm-pi.c.} characterizes the $\pi$-complemented rings as those rings which are almost complemented and nearly reduced. Recall that being nearly reduced is equivalent to saying that the only nilpotent element annihilated by an aregular element is 0. The other extreme is to say that every nilpotent element is annihilated by an aregular element. Symbolically, this means that $\eta(R)=\fN(R)$. In this case, we say that $R$ is {\bf roughly reduced}. The proof of the next result is clear.

\begin{proposition}
1) The ring $R$ is roughly reduced if and only if $q(R)$ is roughly reduced.

2) The ring $R$ is roughly complemented if and only if $q(R)$ is roughly complemented.
\end{proposition}

The following lemma shall be useful and is in the same vein as (3) of Lemma \ref{MainLemma}.

\begin{lemma}\label{mlrc}
Suppose that $a$ and $b$ are rough complements. For any $s\in \areg(R)$, $a$ and $sb$ are rough complements.
\end{lemma}

\begin{proof}
So we are given that $ab=0$ and $a+b\in\areg(R)$. Let $s\in\areg(R)$. Clearly, $a(sb)=0$. So we show that $a+sb\in \areg(R)$. Let $x\in R$ satisfy $(a+sb)x\in \fN(R)$. Then for some $n\in\ N$, $a^nx^n=s^nb^nx^n$. Multiplying both sides by $a$ yields that $a^{n+1}x^n=0$, and then by $b$ produces $b^{n+1}s^nx^n=0$. Thus, $(a+b)^{n+1}s^{n+1}x^{n+1}=0$ and so $(a+b)sx\in \fN(R)$. By aregularity of $(a+b)s$, we conclude that $x\in\fN(R)$.
\end{proof}

We have chosen the definition of roughly complemented to say that all elements, even the nilpotents, are roughly complemented. We did so to obtain a sense of symmetry with respect to the $\pi$-complemented condition. Recall that a $\pi$-complemented ring is nearly reduced. So, we demonstrate that a roughly complemented ring is roughly reduced.

\begin{lemma}\label{rc}
Let $R$ be a ring and $x\in \fN(R)$. Then $x$ is roughly complemented if and only if $x\in \eta(R)$.
\end{lemma}

\begin{proof}
Suppose that $x$ is roughly complemented and let $b\in R$ satisfy $xb=0$ and $x+b\in\areg(R)$. Since $x\in \fN(R)$ it follows that $b\in \areg(R)$. Therefore, since $bx=0$, $x\in \eta(R)$.

Conversely, if $x\in \eta(R)$, then there is some $b\in \areg(R)$ such that $xb=0$. Since $x\in \fN(R)$, $x+b\in \areg(R)$; whence $x$ is roughly complemented.
\end{proof}

\begin{proposition}
The ring $R$ is almost complemented and roughly reduced if and only if $R$ is roughly complemented.
\end{proposition}

\begin{proof}
Necessity. Let $a\in R\ssm \fN(R)$. Choose $b\in R$ such that $ab\in\fN(R)$ and $a+b\in\areg(R)$. Let $n\in\N$ satisfy $(ab)^n=0$. Now, since $R$ is roughly reduced and $ab\in\fN(R)$, there is some $s\in\areg(R)$ such that $s(ab)=0$. We claim that $a+sb\in\areg(R)$ from which we can conclude that $a$ is roughly complemented. To that end, let $z\in R$ for which $(a+sb)z\in\fN(R)$. Let $k\in\N$ be so that $(az)^k+(sbz)^k=(az+sbz)^k=0$; whence $a^kz^k=-s^kb^kz^k$. Multiplying both sides by $a$ yields that $a^{k+1}z^k=0$.

Now,
\begin{eqnarray*}
 [(a+b)(sz)]^{k+1}  &=& (asz+bsz)^{k+1} \\
   &=& (asz)^{k+1}+(bsz)^{k+1} \\
   &=& 0.
\end{eqnarray*}

Since $a+b\in\areg(R)$ it follows that $sz\in\fN(R)$. Since $s\in\areg(R)$, $z\in\fN(R)$.
\vspace{.1in}

Sufficiency. Obviously if $R$ is roughly complemented, then $R$ is almost complemented. A consequence of Lemma \ref{rc} is that a roughly complemented ring is roughly reduced.
\end{proof}
\vspace{.1in}

Turning to $T=S+I$ rings we consider the case when $I$ is nil.

\begin{proposition}
Suppose $T=S+I$ with $I$ a nil ideal. Then, $T=S+I$ is roughly complemented if and only if for each $s\in S$ and $i\in I$, there are $t\in S$, $j\in I$ such that $st=0$, $s+t\in\areg(S)$, and $(s+i)j=it$. In particular, if $T$ is roughly complemented, then $S$ is roughly complemented.
\end{proposition}

\begin{proof}
The second claim clearly follows from the first.
\vspace{.1in}

Suppose that $T$ is roughly complemented and let $s\in S$. Then there is some $t+j\in T$ such that $s(t+j)=0$ and $s+t+j\in\areg(T)$. Now, since $S\cap I=\{0\}$, it follows that $st=-sj=0$.  Since $j\in\fN(T)$, $s+t\in\areg(T)$ and hence $s+t\in\areg(S)$. Therefore, $S$ is roughly complemented.

Let $s\in S$, $i\in I$. By hypothesis, there is a $t\in S$, $j\in I$ such that $(s+i)(t+j)=0$ and $s+i+t+j\in\areg(T)$. It follows that $s+t\in \areg(S)$. Also, $0=-st=sj+it+ij$. Thus, the elements $t\in S$, $-j\in I$ satisfy the desired conditions.
\vspace{.1in}

Conversely, for $s+i\in T$, there are $t\in S$, $j\in I$ such that $s+t\in\reg(S)$ and $(s+i)=it$. We leave it to the interested reader to show that $(s+i)(t+j)=0$ and $s+i+t+j\in \areg(T)$.
\end{proof}

In the case of the trivial extension we can say more.  Recall that the $R$-module $M$ is called {\bf torsion} if for every $m\in M$ there is some $r\in \reg(R)$ such that $rm=0$. We say that $M$ is {\bf atorsion} if for each $m\in M$ there is some $r\in\areg(R)$ such that $rm=0$.

\begin{proposition}
Suppose $T=S+I$ with $I$ a nil ideal of $T$. If $S$ is roughly complemented and $I$ is an atorsion $S$-module, then $T$ is roughly complemented. Furthermore, if $I^2=0$, then the converse holds. In particular, $R\propto M$ is roughly complemented if and only if $R$ is roughly complemented and $M$ is an atorsion-module.
\end{proposition}

\begin{proof}
Suppose $S$ is roughly complemented and $I$ is an atorsion $S$-module. Let $s+i\in T$. Choose $r\in S$ such that $rs=0$ and $s+r\in\areg(S)$. Choose $u\in \areg(S)$ such that $ui=0$. We know that $s+ur\in \areg(S)$ by Lemma \ref{mlrc}, hence $s+ur\in\areg(T)$ by Lemma \ref{aregnil}. Furthermore, $(s+i)+ur\in \areg(T)$. Since $(s+i)(ur)=0$ it follows that $T$ is roughly complemented.
\vspace{.1in}

Next, assume that $I^2=0$ and that $T$ is roughly complemented. Let $i\in I$. There is some $s+j\in T$ such that $(s+j)i=0$ and $s+j+i\in\areg(T)$. Since $i,j\in\fN(T)$, $s\in \areg(T)$, and so $s\in\areg(S)$. Also, observe that $si=ji=0$, by hypothesis. Therefore, $I$ is an atorsion $S$-module. Moreover, one can easily see that $S$ is roughly complemented.
\vspace{.1in}

\end{proof}

\begin{example}
Let $R=\Z\propto (\Z\times \Z_5)$. Since $\Z\times \Z_5$ as a $\Z$-module is neither a torsion nor torsion-free, it follows that $R$ is almost complemented and neither roughly complemented nor $\pi$-complemented.
\end{example}

\begin{example}
Let $R$ be a roughly complemented ring and $M$ a torsion $R$-module. Then $S=R\propto M$ is a roughly complemented ring. Observe that since $M$ is an ideal of $S$, the ring $R \cong (R\propto M)/M)$ is an $S$-module, and for any ideal $J$ of $R$, $J$ is also an $S$-module, whence so is $R/J$.
\vspace{.1in}

So let $R=\Z$, $M=\Z_5$, and $S=\Z\propto \Z_5$. Observe that $\Z_5$ is an $S$-module; it is an atorsion $S$-module that is not a torsion $S$-module. Moreover, it follows that $S\propto \Z_5$ is roughly complemented.

\end{example}
\vspace{.3in}

\section{Final Thoughts and Some Questions}
\begin{example}
In general, the homomorphic image of a complemented ring need not even be reduced let alone complemented. The same can be said for the semi-complemented condition. Consider the ring $\Z_{12}\cong \Z_3\times \Z_4$. Since this ring is not reduced yet contains non-trivial idempotents, it is not semi-complemented. But it is almost complemented.
Interestingly, the question of whether a homomorphic image of an almost complemented ring is again almost complemented leads to the situation of starting with a complemented ring $R$ and trying to determine whether $R/I$ is complemented for each radical ideal $I$. Observe that the integers $R=\Z$ are a ring with this property. We do not know the answer in general and leave this for another time.

\end{example}
\vspace{.3in}
{\bf Questions}
\begin{enumerate}[label={\rm \arabic*.}, nolistsep]
\item Is the homomorphic image of an almost complemented ring again almost complemented. This is equivalent to asking whether a reduced homomorphic image of a complemented ring is complemented.
\item If $R$ is complemented, then is $R[[x]]$ also complemented?
\item When is an infinite product of rings almost complemented or roughly complemented?
\item When is $R[[x]]$ semi-complemented?
\end{enumerate}
\vspace{.1in}

{\bf Acknowledgement} This work was partially supported by a grant from NSERC of Canada (\#2022-03783 to Yiqiang Zhou). We would like to thank James Branca on his help with setting up some of the diagrams used in the article.

\end{document}